\documentclass[oneside,reqno,12pt]{amsart} 

\usepackage{amsmath,amsthm,amscd}
\usepackage{amssymb,amsfonts} 
\usepackage{graphicx}
\usepackage{epstopdf}
\usepackage{fullpage}
\usepackage{caption}
\usepackage{subcaption}

\theoremstyle{definition}
\newtheorem{definition}{Definition}

\theoremstyle{plain}
\newtheorem{theorem}[definition]{Theorem}
\newtheorem{proposition}[definition]{Proposition}
\newtheorem{corollary}[definition]{Corollary}
\newtheorem{lemma}[definition]{Lemma}

\date{}
\begin{document}

\newcommand\Scal{{\mathcal S}}
\newcommand\F{{\mathcal F}}

\newcommand\R{{\mathbb R}}
\newcommand\C{{\mathbb C}}
\newcommand\Z{{\mathbb Z}}
\newcommand\N{{\mathbb N}}

\

\newcommand\dd{\partial}
\newcommand\ddp{d^\prime}
\newcommand\ddpp{d^{\prime\prime}}
\newcommand\lp{\ell^\prime}
\newcommand\lpp{\ell^{\prime\prime}}

\newcommand\Sh{{\widehat S}}

\newcommand\gge{\succeq}
\newcommand\ggr{\succ}

\newcommand\Span{\hbox{{\rm Span}}\;}
\newcommand\dimm{\hbox{{\rm dim}}\;}

\makeatletter
\newcommand{\LeftEqNo}{\let\veqno\@@leqno}
\makeatother

\title
[Bihomogeneous symmetric functions]
{Bihomogeneous symmetric functions}
\author{Yuly Billig}
\address{School of Mathematics and Statistics, Carleton University, Ottawa, Canada}
\email{billig@math.carleton.ca}

\


\

\

\begin{abstract} {We consider two natural gradings on the space of symmetric functions: by degree and by length. We introduce a differential operator $T$ that leaves the components of this double grading invariant and exhibit a basis of bihomogeneous symmetric functions in which this operator is triangular. This allows us to compute the eigenvalues of $T$, which turn out to be non-negative integers.}
\end{abstract}
\let\thefootnote\relax\footnote{2010 Mathematics Subject Classification: 05E05, 35P05.}

\maketitle
 
\

\


 Consider the following second order differential operator
\begin{displaymath}
 T = \frac{1}{2} \sum\limits_{ {\scriptstyle a+b = p+ q} \atop {\scriptstyle a,b,p,q \geq 1}}
x_a x_b \frac{\dd}{\dd x_p} \frac{\dd}{\dd x_q} 
\end{displaymath}
acting on the the Fock space $\F = \R [x_1, x_2, x_3, \ldots ]$. 

The Fock space $\F$ has two natural $\Z$-gradings, by degree, with deg$(x_k) = k$, and
by length, with len$(x_k) = 1$. 
Note that the two gradings are given by the eigenspaces of two operators, whose eigenvalues are degree and length respectively:
\begin{displaymath}
\sum\limits_{k\geq 1} k x_k  \frac{\dd}{\dd x_k} \quad\text{and}\quad
\sum\limits_{k\geq 1} x_k  \frac{\dd}{\dd x_k}.
\end{displaymath}

We can decompose $\F$ into a direct sum according to these gradings:
\begin{displaymath} 
\F = \mathop\oplus\limits_{d \geq \ell \geq 0} \F (d, \ell), 
\end{displaymath}
where $\F (d, \ell)$ is the span of all monomials of degree $d$ and length $\ell$. 

Since operator $T$ preserves both gradings, the subspaces $\F (d, \ell)$ are $T$-invariant. This work started with an observation that the eigenvalues of $T$ on $\F (d, \ell)$ appeared to be non-negative integers. For example, the spectrum of $T$ on $\F (12,4)$ is
\begin{displaymath} 
[1, 3, 3, 5, 6, 7, 7, 10, 10, 10, 13, 15, 17, 19, 30].
\end{displaymath}
 The goal of this note is to shed light on the pattern of the eigenvalues of $T$.

The space $\F$ may be viewed as the space of symmetric functions in variables
$t_1, t_2, t_3, \ldots$. Recall that power symmetric functions are
\begin{displaymath}
p_k = \sum_{i \geq 1} t_i^k, \quad k\geq 1.
\end{displaymath}

Then the algebra of symmetric functions is freely generated by $p_1, p_2, p_3, \ldots$
and may be identified with $\F$ via $x_k = \frac{p_k}{k}$. We refer to \cite{M} for the 
basic properties of symmetric functions that we review here.

Recall also the definitions of the elementary symmetric functions $e_k$ and complete 
symmetric functions $h_k$:
\begin{displaymath}
e_k = \sum\limits_{i_1 < i_2 < \ldots < i_k} t_{i_1} t_{i_2} \ldots t_{i_k}, 
\quad  
h_k = \sum\limits_{i_1 \leq i_2 \leq \ldots \leq i_k} t_{i_1} t_{i_2} \ldots t_{i_k}.
\end{displaymath}
Introducing the generating series 
\begin{displaymath}
e(z) = 1 + \sum_{k=1}^\infty e_k z^k, \quad h(z) = 1 + \sum_{k=1}^\infty h_k z^k,
\end{displaymath}
one can relate elementary and complete symmetric functions to power symmetric functions via
\begin{equation}
\label{hh}
h(z) = \exp \left( \sum_{k=1}^\infty \frac{p_k}{k} z^k \right)
= \exp \left( \sum_{k=1}^\infty x_k z^k \right),
\end{equation}
\begin{equation}
\label{ee}
e(z) = \exp \left( \sum_{k=1}^\infty (-1)^{k+1} \frac{p_k}{k} z^k \right)
= h(-z)^{-1}.
\end{equation}

 With respect to the two gradings on the Fock space, the power symmetric function $p_k$
has degree $k$ and length $1$. It follows from (\ref{hh}) and (\ref{ee}) that $h_k$ and $e_k$ have both degree $k$, but these functions are not homogeneous with respect to the grading given by length. Let us consider decompositions of elementary and complete symmetric functions into bihomogeneous components.
In order to do that, we introduce the generating series 
\begin{displaymath} 
h (r, z) = \exp \left( r \sum_{j=1}^\infty x_j z^j \right) = 1 + \sum_{d \geq \ell \geq 1}
g(d, \ell) r^\ell z^d .
\end{displaymath}
Then
\begin{displaymath}
h_k = \sum_{\ell \leq k} g(k, \ell), \quad 
e_k = \sum_{\ell \leq k} (-1)^{k+\ell} g(k, \ell).
\end{displaymath}
Note that $g(d, \ell) \in \F (d, \ell)$. We shall see below that $g(d, \ell)$ is an eigenfunction for $T$ which corresponds to the dominant eigenvalue on $\F (d, \ell)$.
Our plan is to calculate the spectrum of $T$ by constructing a bihomogeneous basis of the algebra of symmetric functions which consists of the products of functions $g(d, \ell)$.

\

We begin by showing that $T$ is diagonalizable.

\begin{proposition} 
\label{selfadj}
Operator $T$ is diagonalizable on $\F$ with real non-negative eigenvalues.
\end{proposition}

\begin{proof} 
Introduce a positive-definite scalar product on $\R [x]$ with $\left< x^n, x^m \right> = n! \delta_{n,m}$. It is easy to check that this scalar product satisfies
$$\left< \frac{d}{dx} f(x), g(x) \right> = \bigg< f(x), x g(x) \bigg> .$$
Viewing $\F$ as a tensor product of infinitely many copies of $\R[x]$,
$\F = \R [x_1] \otimes \R [x_2] \otimes \ldots$, we obtain a positive-definite scalar product on $\F$ for which $\frac{\dd}{\dd x_k}$ is adjoint to multiplication by $x_k$.
Then for $f, g \in \F$
\begin{displaymath} 
\left<  \sum\limits_{ a+b = p+ q} 
x_a x_b \frac{\dd}{\dd x_p} \frac{\dd}{\dd x_q} f , g \right> =
\left<  f, \sum\limits_{ a+b = p+ q} 
x_q x_p \frac{\dd}{\dd x_b} \frac{\dd}{\dd x_a} g \right> 
\end{displaymath}
and hence $T$ is a self-adjoint operator. Thus $T$ is diagonalizable on each invariant subspace $\F (d, \ell)$ with real eigenvalues.

Since
\begin{displaymath} 
\left< Tf, f \right> = \frac{1}{2} \sum_{n=2}^\infty 
\left< \sum_{p+q = n}  \frac{\dd}{\dd x_p} \frac{\dd}{\dd x_q} f  ,
\sum_{a+b = n}   \frac{\dd}{\dd x_b} \frac{\dd}{\dd x_a} f \right> \geq 0, 
\end{displaymath}
the eigenvalues of $T$ are non-negative.
\end{proof}

\

\begin{corollary} 
There is an orthonormal basis (with respect to the scalar product introduced in the proof of Proposition \ref{selfadj}) of $\F$, consisting of the eigenfunctions of $T$.
\end{corollary}

\

The dimension of $\F (d, \ell)$ is equal to the number of partitions of $d$ with exactly $\ell$ parts. Each such partition may be presented as a Young diagram $\Lambda$,
for example the following diagram
\begin{figure}[h]
\includegraphics{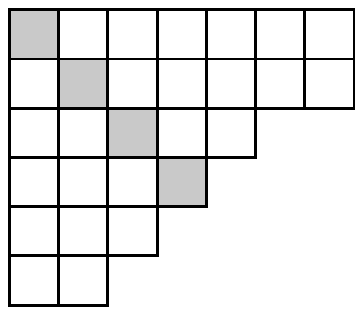} 
\caption{}
\label{yng}
\end{figure}
represents a partition $28 = 7 + 7 + 5 + 4 + 3 + 2$ with $d = 28$ and $\ell = 6$. Parameter $\ell$ is the number of rows in $\Lambda$, while $d$ is the total number of boxes in $\Lambda$. 

Let $k$ be the number of the diagonal boxes in $\Lambda$ (shaded boxes in Fig.~\ref{yng}). For each diagonal box, consider its {\it hook}, the boxes in its row to the right of the diagonal box, the boxes in its column below the diagonal box, together with the diagonal box itself. Also consider the {\it leg} of a diagonal box, the boxes in its column together with the diagonal box itself.
\begin{figure}
\centering
\begin{subfigure}{1.5in}
\includegraphics{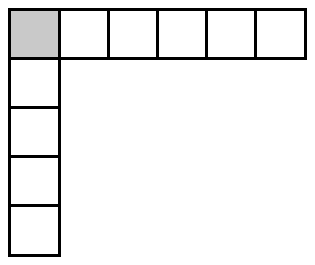} 
\caption{Hook}
\end{subfigure}
\hbox{\hskip 0.9in}
\begin{subfigure}{1.5in}
\includegraphics{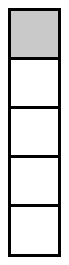} 
\caption{Leg}
\end{subfigure} 
\caption{}
\label{hookandleg}
\end{figure}

In Fig.~\ref{hookandleg} we show the hook and the leg corresponding to the second diagonal box in the above Young diagram $\Lambda$.

The {\it hook number} $d_i$ and the {\it leg number} $q_i$ of a diagonal box are the numbers of boxes in its hook and leg respectively. For the Young diagram $\Lambda$ above, the hook and the leg numbers $(d_i, q_i)$, $i =1, \ldots, k$, are $(12, 6), (10, 5), (5,3), (1,1)$.

Note that our definition of the leg number is not quite standard, usually the diagonal box is not included in its leg.

The hook and leg numbers satisfy $\sum\limits_{i=1}^k d_i = d$, \ $q_1 = \ell$, \ 
$d_i - q_i > d_{i+1} - q_{i+1}$ for $i < k$.

To each diagonal box we also assign its {\it leg increment} $\ell_i = q_i - q_{i+1}$, where $q_{k+1}$ is taken to be $0$. For the Young diagram in Fig.~1, $\ell_1 = 1$, 
$\ell_2 = 2$, $\ell_3 = 2$, $\ell_4 = 1$. Leg increments satisfy
\begin{equation}
\label{cond}
 \sum_{i=1}^k \ell_i = \ell, \ \ d_i > d_{i+1} + \ell_i \hbox{\rm \ \ for \ } i<k, \ \  d_k \geq \ell_k, \  \  \ell_1, \ldots, \ell_k \geq 1.
\end{equation}
Note that there is a bijective correspondence between Young diagrams with $\ell$ rows and sequences $(d_1, \ell_1), \ldots, (d_k, \ell_k)$ satisfying (\ref{cond}).

\

\begin{theorem} 
\label{basis}
The set $S(d,\ell)$ of polynomials
$ g(d_1, \ell_1) g(d_2, \ell_2) \ldots g(d_k, \ell_k)$ satisfying conditions
\begin{displaymath}
\sum_{i=1}^k d_i = d, \ \ \sum_{i=1}^k \ell_i = \ell,
\ \ d_i > d_{i+1} + \ell_i  \ {\it \ for \ }  i<k, \ \ d_k \geq \ell_k , 
\ \ \ell_1, \ldots, \ell_k \geq 1,
\end{displaymath}
forms a basis of $\F (d, \ell)$ for $d \geq \ell \geq 1$.
\end{theorem}

\

 Let $\ell^\prime, \ell^{\prime\prime} \geq 1$. The products $g(\ddp,\lp) g(\ddpp, \lpp)$
with $\ddpp + \lp \geq \ddp \geq \ddpp$ will be called {\it irregular}, while the products with $\ddp > \ddpp + \lp$ will be called {\it regular}. Here we set $g(0,0) = 1$ and consider $g(d, \ell) g(0, 0)$ with $d \geq \ell \geq 1$, to be a regular product.

The proof of Theorem \ref{basis} will be based on the following

\begin{lemma} 
\label{twoprod}
Every irregular product $g(\ddp,\lp) g(\ddpp, \lpp)$
with $\ddpp + \lp \geq \ddp \geq \ddpp$, \ $\lp, \lpp \geq 1$, is a linear combination of regular products  $g(d_1, \ell_1) g(d_2, \ell_2)$, with 
$d_1 + d_2 = \ddp + \ddpp$, \ $\ell_1 + \ell_2 = \lp + \lpp$,
where either $d_1 > \ddp$ or $d_1 = \ddp$ and $\ell_1 < \lp$.
\end{lemma}

\begin{proof} 
We will consider the case when $d = \ddp + \ddpp$ is odd and $\ell = \lp + \lpp$ is even, $d = 2n+1$, $\ell = 2m$. The cases of other parities are analogous.
We can write $\ddp = n + p$, \  $\ddpp = n - p + 1$, \ $2p - 1 \leq \lp \leq 2m - 1$,
\ $\lpp = 2m - \lp$, \  $1 \leq p \leq m$.

We will use a decreasing induction in $p$. As a basis of induction we may choose 
$p = m + 1$, in which case all products are regular and there is nothing to prove. Let us carry out the step of induction. We assume that the claim of the Lemma holds for irregular products $g(d_1, \ell_1) g(d_2, \ell_2)$ with $d_1 > \ddp$.

Consider the generating function
\begin{equation}
\label{geniden}
\left[ \prod_{i=1}^{2p-1} \left( z \frac{d}{dz} + p - n - i \right)
\prod_{{\scriptstyle j=2p-1}\atop{\scriptstyle j\neq \lp}}^{2m-1} \left( r \frac{d}{dr} - j \right) h(r,z) \right]
h(-r, z) . 
\end{equation}
Since the total number of derivatives is $\ell-1$, (\ref{geniden}) is in fact a polynomial in $r$ of degree
$\ell - 1$. Hence the coefficient at $z^d r^\ell$ in (\ref{geniden}) is equal to $0$. This yields
an identity
\begin{equation}
\label{algiden}
\sum_{{\scriptstyle d_1 + d_2 = d} \atop {\scriptstyle \ell_1 + \ell_2 = \ell}} (-1)^{\ell_2}
\prod_{i=1}^{2p-1} \left( d_1 + p - n - i \right)
\prod_{{\scriptstyle j=2p-1}\atop{\scriptstyle j\neq \lp}}^{2m-1} \left( \ell_1 - j \right) 
g(d_1, \ell_1) g(d_2, \ell_2) = 0.
\end{equation}
The terms in (\ref{algiden}) with $n-p+1 \leq d_1 \leq n+p-1$ vanish, thus the only terms that occur have $d_1 \geq \ddp$ or $d_2 > \ddp$. If we look at the terms in (\ref{algiden}) with $d_1 = \ddp$, all such irregular terms will vanish, except for  $g(\ddp,\lp) g(\ddpp, \lpp)$. Thus we can use (\ref{algiden}) to express $g(\ddp,\lp) g(\ddpp, \lpp)$ as a linear combination of regular products and those irregular products for which the claim of the Lemma holds by the induction assumption. All regular products in the expansion of $g(\ddp,\lp) g(\ddpp, \lpp)$ will have
$d_1 > \ddp$ or $d_1 = \ddp$ and $\ell_1 < \lp$. This completes the proof of the Lemma.
\end{proof}

\

Let us order the set of pairs $(d, \ell)$ as follows: $(d_1, \ell_1) \ggr (d_2, \ell_2)$ if either $d_1 > d_2$ or $d_1 = d_2$ and $\ell_1 < \ell_2$. Consider the set
$\Sh (d, \ell)$ of ordered products
$g(d_1, \ell_1) g(d_2, \ell_2) \ldots g(d_k, \ell_k)$ with
$(d_1, \ell_1) \gge (d_2, \ell_2) \gge \ldots \gge (d_k, \ell_k)$,
\ $\sum\limits_{i=1}^k d_i = d$, \ $\sum\limits_{i=1}^k \ell_i = \ell$, \ $\ell_i \geq 1$. Introduce a lexicographic order on $\Sh (d, \ell)$:
$$g(\ddp_1, \lp_1) g(\ddp_2, \lp_2) \ldots g(\ddp_k, \lp_k) \ggr
g(\ddpp_1, \lpp_1) g(\ddpp_2, \lpp_2) \ldots g(\ddpp_k, \lpp_k)$$
if for some $m$, $(\ddp_i, \lp_i) = (\ddpp_i, \lpp_i)$ for $i = 1, \ldots, m-1,$
and $(\ddp_m, \lp_m) \ggr (\ddpp_m, \lpp_m)$.

 Now we can prove Theorem \ref{basis}. The set $\Sh (d, \ell)$ clearly spans the space 
$\F (d, \ell)$ since $g(p, 1) = x_p$ and $\F (d, \ell)$ is spanned by monomials. It follows from Lemma \ref{twoprod} that every product from $\Sh (d, \ell)$ which is not in $S(d, \ell)$ may be expressed as a linear combination of the elements of $\Sh (d, \ell)$ which are greater in the lexicographic order. By induction with respect to this ordering we conclude that
\begin{displaymath}
\F (d, \ell) = \Span \Sh (d, \ell) = \Span S(d,\ell)  .
\end{displaymath}
However, elements of $S(d, \ell)$ are parametrized by Young diagrams with $d$ boxes and $\ell$ rows. Hence $|S(d, \ell)| = \dimm \F(d, \ell)$ and $S(d, \ell)$ is a basis of $\F(d, \ell)$. This completes the proof of Theorem \ref{basis}.

\

Let us compute the eigenvalues of the differential operator $T$.

\

\begin{theorem} Eigenvalues of the operator
\begin{displaymath}
 T = \frac{1}{2} \sum\limits_{ {\scriptstyle a+b = p+ q} \atop {\scriptstyle a,b,p,q \geq 1}}
x_a x_b \frac{\dd}{\dd x_p} \frac{\dd}{\dd x_q} 
\end{displaymath}
on $\F (d, \ell)$, $\ell \geq 1$, are parametrized by sequences
$(d_1, \ell_1), (d_2, \ell_2), \ldots, (d_k, \ell_k)$ with 
\begin{displaymath}
d_i > d_{i+1} + \ell_i, \ i = 1, \ldots, k-1, \ \ d_k \geq \ell_k,
\ \ \sum\limits_{i=1}^k d_i = d, \ \ \sum\limits_{i=1}^k \ell_i = \ell, \ \ 
\ell_1, \ldots, \ell_k \geq 1.
\end{displaymath}
The corresponding eigenvalue is
\begin{displaymath}
\lambda = \frac{1}{2} \sum_{i=1}^k (\ell_i -1) (2 d_i - \ell_i) .
\end{displaymath}
\label{main}
\end{theorem}

\begin{proof}
We are going to show that the matrix of the operator $T$ is upper-triangular in the basis $S(d, \ell)$ ordered by $\ggr$. Then the spectrum of $T$ is given by the diagonal of this matrix.

Consider the generating functions
\begin{displaymath}
X_i = r_i \sum_{j=1}^\infty x_j z_i^j 
\end{displaymath}
and
\begin{displaymath}
E = \exp \left( \sum_{i=1}^\infty X_i \right) =
\exp \left(  \sum_{i=1}^\infty r_i \sum_{j=1}^\infty x_j z_i^j \right) .
\end{displaymath}
The product $g(d_1, \ell_1) \ldots g(d_k, \ell_k)$ is the coefficient at
$z_1^{d_1} \ldots z_k^{d_k} r_1^{\ell_1} \ldots r_k^{\ell_k}$ in $E$. Let us apply operator $T$ to the generating function $E$:
\begin{gather*}
TE = \frac{1}{2} \sum_{n=2}^\infty \left( \sum_{a+b = n} x_a x_b \right)
\sum_{p+q = n} \frac{\dd}{\dd x_p}\frac{\dd}{\dd x_q} E 
\end{gather*}
\begin{gather*}
=   \frac{1}{2} \sum_{n=2}^\infty \left( \sum_{a+b = n} x_a x_b \right)
\sum_{p+q = n} \left( \sum_{i=1}^\infty r_i z_i^p \right) 
\left( \sum_{j=1}^\infty r_j z_j^q \right) E 
\end{gather*}
\begin{gather*}
=   \frac{1}{2}   \sum_{i=1}^\infty r_i^2   
\sum_{n=2}^\infty \left( \sum_{a+b = n} x_a x_b \right)
\sum_{p+q = n} z_i^{p+q} E \\
+  \sum_{i<j} r_i r_j \sum_{n=2}^\infty \left( \sum_{a+b = n} x_a x_b \right)
\sum_{p+q = n} z_i^p z_j^q E 
\end{gather*}
\begin{gather*}
=   \frac{1}{2}   \sum_{i=1}^\infty r_i^2   
\sum_{n=2}^\infty \left( \sum_{a+b = n} x_a x_b \right)
(n-1) z_i^n E \\
+  \sum_{i<j} r_i r_j \sum_{n=2}^\infty \left( \sum_{a+b = n} x_a x_b \right)
\left( 1 - \frac{z_j}{z_i} \right)^{-1} \left( z_j z_i^{n-1} - z_j^n \right) E 
\end{gather*}
\begin{gather*}
=   \frac{1}{2}   \sum_{i=1}^\infty 
\left[ \left( z_i \frac{d}{dz_i} - 1 \right) r_i^2 \sum_{a, b \geq 1} x_a x_b z_i^{a + b} 
\right] E \\
 - \sum_{i<j} \left( 1 - \frac{z_j}{z_i} \right)^{-1} \frac{r_i}{r_j} r_j^2   
\sum_{a,b \geq 1} \left( x_a x_b z_j^{a+b} \right) E
 + \sum_{i<j} \frac{z_j}{z_i} \left( 1 - \frac{z_j}{z_i} \right)^{-1} \frac{r_j}{r_i} r_i^2   
\sum_{a,b \geq 1} \left( x_a x_b z_i^{a+b} \right) E 
\end{gather*}
\begin{gather*}
=   \frac{1}{2}   \sum_{i=1}^\infty 
\left[ \left( z_i \frac{d}{d z_i} -1 \right) X_i^2 \right] E
 - \sum_{i<j} \left( 1 - \frac{z_j}{z_i} \right)^{-1} \frac{r_i}{r_j} X_j^2 E
 + \sum_{i<j} \frac{z_j}{z_i} \left( 1 - \frac{z_j}{z_i} \right)^{-1} \frac{r_j}{r_i} 
X_i^2  E 
\end{gather*}
\begin{gather*}
=  \sum_{i=1}^\infty 
\left[ \left( z_i \frac{d}{d z_i} - \frac{1}{2} \right) X_i \right] X_i E
 - \sum_{i<j} \left( 1 - \frac{z_j}{z_i} \right)^{-1} \frac{r_i}{r_j} X_j^2 E
 + \sum_{i<j} \frac{z_j}{z_i} \left( 1 - \frac{z_j}{z_i} \right)^{-1} \frac{r_j}{r_i} 
X_i^2  E 
\end{gather*}
\begin{gather*}
=  \sum_{i=1}^\infty 
\left[ \left( z_i \frac{d}{d z_i} - \frac{1}{2} \right) X_i \right] r_i \frac{d}{dr_i} E \\
 - \sum_{i<j} \left( 1 - \frac{z_j}{z_i} \right)^{-1} \frac{r_i}{r_j} X_j
r_j \frac{d}{dr_j} E
 + \sum_{i<j} \frac{z_j}{z_i} \left( 1 - \frac{z_j}{z_i} \right)^{-1} \frac{r_j}{r_i} 
X_i r_i \frac{d}{dr_i}  E 
\end{gather*}
\begin{gather*}
=  \sum_{i=1}^\infty \left( r_i \frac{d}{dr_i} - 1\right)
\left[ z_i \frac{d}{d z_i} X_i - \frac{1}{2} X_i  \right]  E \\
 - \sum_{i<j} \left( 1 - \frac{z_j}{z_i} \right)^{-1} r_j \frac{d}{dr_j}
\frac{r_i}{r_j} X_j E
 + \sum_{i<j} \frac{z_j}{z_i} \left( 1 - \frac{z_j}{z_i} \right)^{-1}  r_i \frac{d}{dr_i}
\frac{r_j}{r_i} X_i  E 
\end{gather*}
\begin{gather*}
=  \sum_{i=1}^\infty \left( r_i \frac{d}{dr_i} - 1\right)
\left( z_i \frac{d}{d z_i}  - \frac{1}{2} r_i \frac{d}{dr_i}  \right)  E \\
 - \sum_{i<j} \left( 1 - \frac{z_j}{z_i} \right)^{-1} r_j \frac{d}{dr_j}
\frac{r_i}{r_j}  r_j \frac{d}{dr_j} E
 + \sum_{i<j} \frac{z_j}{z_i} \left( 1 - \frac{z_j}{z_i} \right)^{-1}  r_i \frac{d}{dr_i}
\frac{r_j}{r_i}  r_i \frac{d}{dr_i}  E .
\end{gather*}
To get the formula for the action of $T$ on the elements of $S(d, \ell)$, we extract 
the coefficient at  $z_1^{d_1} \ldots z_k^{d_k} r_1^{\ell_1} \ldots r_k^{\ell_k}$ in $TE$:
\begin{multline*}
T g(d_1, \ell_1) g(d_2, \ell_2) \ldots g(d_k, \ell_k) \\
= \sum_{i=1}^k (\ell_i -1) \left( d_i - \frac{\ell_i}{2} \right)  
g(d_1, \ell_1) g(d_2, \ell_2) \ldots g(d_k, \ell_k) \\
- \sum_{i<j} \sum_{p=0}^\infty \ell_j (\ell_j + 1) 
g(d_1, \ell_1) \ldots g(d_i+p, \ell_i - 1) \ldots g(d_j - p, \ell_j + 1)  \ldots g(d_k, \ell_k) \\
+ \sum_{i<j} \sum_{p=1}^\infty \ell_i (\ell_i + 1) 
g(d_1, \ell_1) \ldots g(d_i+p, \ell_i + 1) \ldots g(d_j - p, \ell_j - 1)  \ldots g(d_k, \ell_k) .
\end{multline*}

The first part in the above expression yields the diagonal part of $T$ with the eigenvalue
$\lambda = \frac{1}{2} \sum\limits_{i=1}^k (\ell_i -1) (2 d_i - \ell_i)$, while the last two sums when expanded in the basis $S(d, \ell)$ applying Lemma \ref{twoprod} whenever necessary,
only contain terms that are strictly greater than  
\break
$g(d_1, \ell_1) g(d_2, \ell_2) \ldots g(d_k, \ell_k)$ with respect to the lexicographic order
$\ggr$. This completes the proof of the Theorem.
\end{proof}

\

It follows from the proof of Theorem \ref{main} that $g(d, \ell)$ is the eigenfunction for the operator $T$ with the eigenvalue $\lambda = \frac{1}{2} (\ell -1) (2 d - \ell)$, which is the dominant eigenvalue on $\F (d, \ell)$.

 Let us point out that $0$ is an eigenvalue of $T$ on $\F (d, \ell)$ if and only if $d \geq \ell^2$. 

 We can obtain the orthogonal basis of eigenfunctions for $T$ in $\F(d,\ell)$ from the ordered basis $S(d, \ell)$ using the Gram-Schmidt procedure.

\

{\bf Acknowledgements.} I thank Nantel Bergeron, Sergey Fomin and Emmanuel Lorin for  their helpful comments.
Support from the Natural Sciences and Engineering Research Council of Canada is gratefully acknowledged.

\

%
%
%
%
%


\begin{thebibliography}{99}

\bibitem{M}
I.~G.~Macdonald,  Symmetric functions and Hall polynomials. 2nd ed., Oxford university press, 1995.

\end{thebibliography}
\end{document}